\documentclass[12pt]{amsart}
\usepackage{amscd,amsmath,amsthm,amssymb,graphics}
\usepackage{amsfonts,amssymb,amscd,amsmath,enumitem,verbatim}

\theoremstyle{plain}
\newtheorem{Theorem}{Theorem}
\newtheorem{Lemma}[Theorem]{Lemma}

\newtheorem{Proposition}[Theorem]{Proposition}

\newtheorem{Conjecture}[Theorem]{Conjecture}

\theoremstyle{definition}
\newtheorem{Definition}[Theorem]{Definition}
\newtheorem{Remark}[Theorem]{Remark}

\newtheorem{Example}[Theorem]{Example}

\usepackage{float}
\usepackage[capitalise]{cleveref}

\crefname{Theorem}{Theorem}{Theorems}
\crefname{Proposition}{Proposition}{Propositions}

\begin{document}
\title[Periodicity of an algorithm for $p$--adic continued fractions]{On the periodicity of an algorithm for $p$--adic continued fractions}
\author {Nadir Murru} 
\address{Department of Mathematics, Università degli Studi di Trento}
\email{nadir.murru@unitn.it}
\author{Giuliano Romeo} 
\address{Department of Mathematical Sciences "Giuseppe Luigi Lagrange", Politecnico di Torino}
\email{giuliano.romeo@polito.it}
\author{Giordano Santilli}
\address{Department of Mathematics, Università degli Studi di Trento}
\email{giordano.santilli@unitn.it}

\begin{abstract}
In this paper we study the properties of an algorithm (introduced in \cite{BII}) for generating continued fractions in the field of $p$--adic numbers $\mathbb Q_p$. First of all, we obtain an analogue of the Galois' Theorem for classical continued fractions. Then, we investigate the length of the preperiod for periodic expansions of square roots. Finally, we prove that there exist infinitely many square roots of integers in $\mathbb Q_p$ that have a periodic expansion with period of length four, solving an open problem left by Browkin in \cite{BII}.  
\end{abstract}

\maketitle
\section{Introduction}

Continued fractions are very important objects both from a theoretical and an applicative point of view. First of all, they allow to characterize rational numbers and quadratic irrationals over the real numbers. Indeed, a continued fraction has a finite expansion if and only if it represents a rational number and it has a periodic expansion if and only if it represents a quadratic irrational. Moreover, they provide the best approximations of real numbers and they are used in applied fields like cryptography.
Thus, it has been natural to introduce them over the field of $p$--adic numbers $\mathbb{Q}_p$, with the aim of reproducing all the good properties of the classical continued fractions and exploiting them for deepening the knowledge of the $p$--adic numbers.
Nevertheless, the definition of a good and satisfying $p$--adic continued fraction algorithm is still an open problem, since it is hard to find one interpretation that retrieves all the results that hold in the real case. 
There have been several attempts trying to emulate the standard definition in the field of real numbers (see \cite{BI, BII, RUB, SCH}). The continued fractions introduced in \cite{RUB} and \cite{SCH} can have either a finite or periodic expansion when representing a rational number (see \cite{BUN} and \cite{LAO}). Moreover, the two algorithms are not periodic for all quadratic irrationals (the characterization of the periodicity for these two algorithms can be found, respectively, in \cite{TIL} and \cite{CVZ}). On the contrary, in \cite{BI}, Browkin defined a $p$--adic continued fraction algorithm terminating in a finite number of steps on every rational number and a full characterization for the periodicity of this algorithm is still missing.
This algorithm, that we call \textit{Browkin I}, works as follows. The $p$--adic continued fraction expansion $[b_0, b_1, \ldots]$ of $\alpha_0\in\mathbb{Q}_p$ provided by \textit{Browkin I} is obtained by iterating the following steps

\begin{equation}\label{Br1}
\begin{cases}
b_n=s(\alpha_n)\\
\alpha_{n+1}=\frac{1}{\alpha_n-b_n},
\end{cases} \quad \forall n \geq 0,
\end{equation}
where $s: \mathbb{Q}_p\rightarrow \mathbb{Q}$ is the function that replaces the role of the floor function in the classical continued fractions over $\mathbb R$. For a $p$--adic number $\alpha=\sum\limits_{n=-r}^{+\infty} a_np^n\in\mathbb{Q}_p$, with $r\in\mathbb{Z}$ and $a_n\in \{-\frac{p-1}{2},\ldots,\frac{p-1}{2}\}$, the function $s$ is defined as
\[s(\alpha)=\sum\limits_{n=-r}^{0} a_np^n\in\mathbb{Q}.\]

This continued fraction is very similar to the one defined in \cite{RUB}, but here the representatives are taken in $\{-\frac{p-1}{2},\ldots,\frac{p-1}{2}\}$ instead of the canonical ones $\{0,\ldots, p-1 \}$. In the same paper, Browkin showed that this algorithm terminates in a finite number of steps if and only if $\alpha\in\mathbb{Q}$. 
More than 20 years later, Browkin defined in \cite{BII} another $p$--adic continued fractions algorithm, using a different ``floor function" in combination with the $s$ function, with the aim of obtaining better results on the behaviour of the algorithm over quadratic irrationals. For $\alpha=\sum\limits_{n=-r}^{+\infty} a_np^n\in\mathbb{Q}_p$, with $r\in\mathbb{Z}$ and $a_n\in \{-\frac{p-1}{2},\ldots,\frac{p-1}{2}\}$, the new floor function is defined as
\[ t(\alpha)=\sum\limits_{n=-r}^{-1}a_np^n.\]
The new algorithm proposed by Browkin, which we call \textit{Browkin II}, works as follows: starting from $\alpha_0\in\mathbb{Q}_p$ the partial quotients of the $p$--adic continued fraction expansion are computed for $n \geq 0$ as
\begin{align} 
\begin{cases}\label{Br2}
b_n=s(\alpha_n) \ \ \ \ \ & \textup{if} \ n \ \textup{even}\\
b_n=t(\alpha_n) & \textup{if} \ n \ \textup{odd}\ \textup{and} \ v_p(\alpha_n-t(\alpha_n))= 0\\
b_n=t(\alpha_n)-sign(t(\alpha_n)) & \textup{if} \ n \ \textup{odd} \ \textup{and} \ v_p(\alpha_n-t(\alpha_n))\neq 0\\
\alpha_{n+1}=\frac{1}{\alpha_n-b_n}.
\end{cases}
\end{align}
Barbero, Cerruti and Murru \cite{BCMI} proved that also this second algorithm terminates in a finite number of steps on each rational number. The same authors proved in \cite{BCMII} that this result is still true when the algorithm is performed using the canonical representatives in $\{0,\ldots, p-1 \}$. In \cite{BCMI} and \cite{BII}, it has been observed that \textit{Browkin II} behaves better than \textit{Browkin I} on quadratic irrationals. In particular, it appears to be periodic on more square roots and in general periods are shorter than those of the expansions produced by \textit{Browkin I}. The periodicity properties of \textit{Browkin I} are well-studied (see, e.g., \cite{BCMII, BEI, BEII, CMT, DEA, WANI, WANII,DEWII}), but few considerations have been developed about \textit{Browkin II}, due to the complexity of this algorithm that uses alternately two different functions.  
However, the study of the periodicity of both the two algorithms is fundamental for understanding how to design more performing $p$--adic continued fraction algorithms.
For these reasons, in this paper we are interested in studying the properties of the periodicity of \textit{Browkin II} $p$--adic continued fractions. In Section \ref{sec:perBro}, we first give a necessary condition for the pure periodicity of \textit{Browkin II}, obtaining similar results to the ones proved by Bedocchi \cite{BEI} for \textit{Browkin I}. Moreover, we prove that if $\sqrt{D}$ exists in $\mathbb Q_p$ and has a periodic expansion by means of \emph{Browkin II}, then the preperiod has length either one or even. In Section \ref{sec:perExp}, we prove that there exist infinitely many square roots of integers that have a periodic expansion of period's length four, answering to a problem left open by Browkin \cite{BII}. We also explicitly provide a family of square roots having such a periodic expansion.  

\section{Preliminaries}
Let us denote by $v_p(\cdot)$ and $|\cdot|_p$, the $p$--adic valuation and the $p$--adic norm. From now on, we will always consider $p$ as an odd prime. We denote by $\frac{A_n}{B_n}$, for all $n\in\mathbb{N}$, the convergents of the continued fraction $[b_0,b_1,\ldots]$, which are
\[\frac{A_n}{B_n}=[b_0,b_1,\ldots,b_n]=b_0+\cfrac{1}{b_1+\cfrac{1}{\ddots + \frac{1}{b_n}}}.\]
We also define the following sets
\[J_p=\left\{\frac{a_0}{p^n}\ \middle| \ n\in\mathbb{N}, \ -\frac{p^{n+1}}{2}< a_0 < \frac{p^{n+1}}{2} \right\}=\mathbb{Z}\left[\frac{1}{p}\right]\cap \left(-\frac{p}{2},\frac{p}{2}\right), \]
and
\[K_p=\left\{\frac{a_0}{p^n}\ \middle| \ n\geq 1, \ -\frac{p^n}{2} < a_0 < \frac{p^{n}}{2}\right\}=\mathbb{Z}\left[\frac{1}{p}\right]\cap \left(-\frac{1}{2},\frac{1}{2}\right).\]
The following lemma is due to Bedocchi.
\begin{Lemma}[Lemma 2.2, \cite{BEI}]\label{LemBedo}
For all $a,b\in J_p$, with $a\neq b$, we have $v_p(a-b)\leq 0$.
\end{Lemma}

In \cite{BI}, Browkin showed that the partial quotients generated by the function $s$ always lie in the set $J_p$. %interval $(-\frac{p}{2},\frac{p}{2})$. 
In the following lemma we prove a similar result for the function $t$ and the set $K_p$.

\begin{Lemma}\label{DisuII}
Let $\alpha=\sum\limits_{n=-r}^{+\infty}a_np^n\in\mathbb{Q}_p$, with $r\in\mathbb{Z}$ and $a_n\in \{-\frac{p-1}{2},\ldots,\frac{p-1}{2} \}$ for all $n\in \mathbb{N}$. Then
\[\Big|t(\alpha)\Big|<\frac{1}{2},\]
where $|\cdot|$ is the Euclidean norm.
\begin{proof}
Using the function $t$ we obtain
\[t(\alpha)=t\Big(\sum\limits_{n=-r}^{+\infty}a_np^n\Big)=\sum\limits_{n=-r}^{-1}a_np^n.\]
If $r\leq 0$ then $|t(\alpha)|=0<\frac{1}{2}$ and the claim holds. When $r > 0$, we have
\begin{align*}
\Big|t(\alpha)\Big|&=\Bigg|\sum\limits_{n=-r}^{-1}a_np^n\Bigg|\leq \frac{p-1}{2}\Bigg|\sum\limits_{n=-r}^{-1}p^n\Bigg|=\\
&=\frac{1}{2} \cdot \frac{(p-1)(1+p+\ldots+p^{r-1})}{p^r}=\\
&=\frac{1}{2}\cdot \frac{p^r-1}{p^r}<\frac{1}{2},
\end{align*}
and the claim is proved.
\end{proof} 
\end{Lemma}

Now we prove the analogue of Lemma \ref{LemBedo} for the set $K_p$.

\begin{Lemma}\label{ab}
For all $a,b\in K_p$, with $a\neq b$, we have $v_p(a-b)< 0$.
\begin{proof}
Let us write $a=\frac{a_0}{p^n}$ and $b=\frac{b_0}{p^m}$, with $v_p(a_0)=v_p(b_0)=0$. We can notice that $n,m\geq 1$ since $v_p(a)$ and $v_p(b)$ are both negative. If $n\neq m$, we may suppose $n>m$ without loss of generality and we get
\begin{align*}
v_p(a-b)&=v_p\Big(\frac{a_0}{p^n}-\frac{b_0}{p^m} \Big)=v_p\Big(\frac{a_0-b_0p^{n-m}}{p^n}\Big)=\\
&=v_p(a_0-b_0p^{n-m})-v_p(p^n)=(n-m)-n=-m<0.
\end{align*}
If $n=m$, then
\[v_p(a-b)=v_p\Big(\frac{a_0}{p^n}-\frac{b_0}{p^n} \Big)=v_p\Big(\frac{a_0-b_0}{p^n}\Big)=v_p(a_0-b_0)-n.\]
Since $|a_0-b_0|<p^n$, necessarily $v_p(a_0-b_0)<n$, hence
\[v_p(a-b)=v_p(a_0-b_0)-n<0,\]
and this concludes the proof.
\end{proof}
\end{Lemma}

%Notice that the sets $J_p$ and $K_p$ coincide with the set of the partial quotients that can be generated by the function $s$ and the function $t$, respectively. 
Lemma \ref{DisuII} and Lemma \ref{ab} allow us to get useful results on the periodicity of \textit{Browkin II}.

\begin{Remark}\label{RemConv}
In \cite{BI} it has been proved that, for \textit{Browkin I}, the valuations of the numerators and the denominators of the convergents are computed as

\begin{equation}
\label{Equ1}
\begin{split}
v_p(A_n)&=v_p(b_0)+v_p(b_1)+\ldots+v_p(b_n),\\
v_p(B_n)&=v_p(b_1)+v_p(b_2)+\ldots+v_p(b_n),
\end{split}
\end{equation}

or, equivalently,
\begin{equation}
\label{Equ2}
\begin{split}
|A_n|_p&=|b_0|_p|b_1|_p\ldots|b_n|_p,\\
|B_n|_p&=|b_1|_p|b_2|_p\ldots|b_n|_p.
\end{split}
\end{equation}

This result is proved by induction using the fact that, for all $n\geq 0$,
\[v_p(b_{n+2}B_{n+1})<v_p(B_n).\]
The latter condition is true also for \textit{Browkin II}, (see \cite{BII}, Lemma 1) and equations (\ref{Equ1}) and (\ref{Equ2}) also hold for \textit{Browkin II}.
\end{Remark}

\section{Periodicity of Browkin's second algorithm}
\label{sec:perBro}
Bedocchi \cite{BEI, BEII} was able to provide some results on the periodicity of Browkin I. First of all he obtained an analogue of the Galois' Theorem for classical continued fractions, which provide a characterization of the pure periodicity by means of reduced quadratic irrational (see, e. g., \cite{Olds}). 
Then he focused on the length of the preperiod and the period for square roots of integers. In particular, he proved that if $\sqrt{D}$, with $D$ integer, has a periodic expansion by means of \textit{Browkin I}, then the preperiod must have length two and the period can not have length one. 

In this section we deepen the study of \textit{Browkin II}. In particular, we obtain an analogue of the Galois' Theorem and we prove that if $\sqrt{D}$, with $D$ integer, has a periodic expansion by means of \textit{Browkin II}, then the preperiod must have length either one or even.

\begin{Remark}
If a $p$--adic number $\alpha$ has a periodic continued fraction, using a standard argument it is possible to see that it is a quadratic irrational \cite[Remark 2.7.5]{BEI}. It means that $\alpha$ is the root of an irreducible polynomial of degree $2$ over $\mathbb{Q}$. We denote by $\overline{\alpha}$, and we call it the conjugate of $\alpha$, the second root of this polynomial, that lies inside $\mathbb{Q}_p$.
\end{Remark}

In the following two theorems, we try to obtain the analogue of the Galois' Theorem also for \textit{Browkin II}.

\begin{Theorem}\label{Bedo1}
If $\alpha\in\mathbb{Q}_p$ has a purely periodic continued fraction expansion $\alpha=[\overline{b_0,\ldots,b_{k-1}}]$ with Browkin II, then
\[ |\alpha|_p=1, \ \ |\overline{\alpha}|_p<1.\]
\end{Theorem}
\begin{proof}
Let us notice that, by construction of \textit{Browkin II}, the period length $k$ must be even. Since, by the pure periodicity,
\[v_p(\alpha)=v_p(b_0)=v_p(b_k)=0,\]
then $|\alpha|_p=1$.
If we set, for all $n\in\mathbb{N}$,
\[\alpha_n=[b_n,b_{n+1},\ldots,b_{n+k-1},\alpha_n]=[b_0',b_1',\ldots,b_{k-1}',\alpha_n], \]
and $\dfrac{A_n'}{B_n'}$ are its convergents, then:
\[B_{k-1}'\alpha_n^2+(B_{k-2}'-A_{k-1}')\alpha_n-A_{k-2}'=0.\]
By \eqref{Equ2},
\[|\alpha_n\overline{\alpha}_n|_p=\left\lvert\frac{A_{k-2}'}{B_{k-1}'}\right\rvert _p=\frac{|b_0'|_p|b_1'|_p\ldots|b_{k-2}'|_p}{|b_1'|_p\ldots|b_{k-2}'|_p|b_{k-1}'|_p}=\frac{|b_0'|_p}{|b_{k-1}'|_p}=\frac{|b_n|_p}{|b_{n+k-1}|_p},\]
from which we get
\[|\overline{\alpha}_n|_p=\frac{1}{|b_{n+k-1}|_p}.\]
Since $k-1$ is odd, then:
\[
\begin{cases}
|\overline{\alpha}_n|_p=1 \ \ \ \ \textup{if} \ n \ \textup{odd}\\ 
|\overline{\alpha}_n|_p<1 \ \ \ \ \textup{if} \ n \ \textup{even},
\end{cases}\]
and, in particular, for $n=0$, $|\overline{\alpha}|_p=|\overline{\alpha}_0|_p<1$.
\end{proof}

In light of Theorem \ref{Bedo1}, it is meaningful to wonder which are (and if there exist) the $p$--adic numbers satisfying the necessary condition for having the pure periodicity.

\begin{Proposition}\label{PureSqrts}
Let $\alpha=a+\sqrt{D}\in\mathbb{Q}_p$, with $a, D\in\mathbb{Z}$, $D$ not a square,
\[\sqrt{D}=a_0+a_1p+a_2p^2+\ldots.\]
Then $|\alpha|_p=1$, $|\overline{\alpha}|_p<1$ if and only if $a\equiv a_0 \mod p$.
\begin{proof}
Let us notice that $a_0\not\equiv 0 \mod p$ and, if  $|\alpha|_p=1$ and $|\overline{\alpha}|_p<1$, then
\[v_p(\overline{\alpha})=v_p(a-a_0-a_1p-a_2p^2-\ldots)>1;\]
it means that $a-a_0\equiv 0 \mod p$, so $a\equiv a_0 \mod p$.\\
Viceversa, if $a\equiv a_0 \mod p$, then $a=a_0+kp$, for some $k\in\mathbb{Z}$. Therefore,
\[v_p(\alpha)=v_p(a+\sqrt{D})=v_p(2a_0+(k+a_1)p+\ldots)=0,\]
since $2a_0\not\equiv 0 \mod p$, for $p\neq 2$; moreover
\[v_p(\overline{\alpha})=v_p(a-\sqrt{D})=v_p((k-a_1)p+\ldots)>0.\]
Hence, in this case, $|\alpha|_p=1$ and $|\overline{\alpha}|_p<1.$
\end{proof}
\end{Proposition}

The converse of Theorem \ref{Bedo1} is not true and the best that one can prove is stated in the following theorem.

\begin{Theorem}\label{LengthII}
Consider $\alpha\in\mathbb{Q}_p$ with periodic \textit{Browkin II} expansion
\[\alpha=[b_0,b_1,\ldots,b_{h-1},\overline{b_h,\ldots,b_{h+k-1}}].\]
If
\[ |\alpha|_p=1, \ \ |\overline{\alpha}|_p<1,\]
then the preperiod length can not be odd.
\begin{proof}
Let us notice that, by the hypotesis and the construction of \textit{Browkin II}, the period length $k$ is even and, for all $j\in\mathbb{N}$,
\begin{align*}
|\alpha|_p&=|\alpha_{2j}|_p=1,\\
v_p(\alpha)&=v_p(b_0)=v_p(b_{2j})=0.
\end{align*}
Moreover,
\begin{align*}
|\overline{\alpha}_0|_p&=|\overline{\alpha}|_p<1,\\
v_p(\overline{\alpha}_0)&=v_p(\overline{\alpha})>0,
\end{align*}
and it follows that:
\begin{align*}
v_p(\overline{\alpha}_1)&=v_p\Big( \frac{1}{\overline{\alpha}_0  - b_0}\Big)=-v_p(\overline{\alpha}_0  - b_0)=-v_p(b_0)=0,\\
v_p(\overline{\alpha}_2)&=v_p\Big( \frac{1}{\overline{\alpha}_1  - b_1}\Big)=-v_p(\overline{\alpha}_1  - b_1)=-v_p(b_1)>0.
\end{align*}
Hence, the $p$--adic absolute value of each complete quotient is, for all $j \in\mathbb{N}$,
\begin{align*}
|\alpha|_p&=1, \ & |\alpha_{2j+1}|_p&>1, \ &|\alpha_{2j}|_p&=1,    \\
|\overline{\alpha}|_p&<1, & |\overline{\alpha}_{2j+1}|_p&=1, &  |\overline{\alpha}_{2j}|_p&<1. 
\end{align*}
Since $\alpha$ has a periodic expansion,
\[\frac{1}{\alpha_{h-1}-b_{h-1}}=\alpha_h=\alpha_{h+k}=\frac{1}{\alpha_{h+k-1}-b_{h+k-1}}.\]
So we obtain:
\[|\alpha_{h-1}-\alpha_{h+k-1}|_p=|b_{h-1}-b_{h+k-1}|_p,\]
and, analogously,
\[|\overline{\alpha}_{h-1}-\overline{\alpha}_{h+k-1}|_p=|b_{h-1}-b_{h+k-1}|_p.\]
By contradiction, if the preperiod length $h$ is odd, both $h-1$ and $h+k-1$ are even. Then $v_p(\overline{\alpha}_{h-1})>0$ and $v_p(\overline{\alpha}_{h+k-1})>0$, so:
\begin{align*}
v_p(b_{h-1}-b_{h+k-1})_p&=v_p(\overline{\alpha}_{h-1}-\overline{\alpha}_{h+k-1})\geq \\ &\geq \min \{v_p(\overline{\alpha}_{h-1}),v_p(\overline{\alpha}_{h+k-1}) \}>0.
\end{align*}
By Lemma \ref{LemBedo}, $b_{h-1}=b_{h+k-1}$ and the claim is proved.
\end{proof}
\end{Theorem}

\begin{Remark}\label{pato}
In the proof of Theorem \ref{LengthII} we obtained that, for $h$ odd, $b_{h+k}=b_{h}$ implies $b_{h+k-1}=b_{h-1}$. This is done by using Lemma \ref{LemBedo} of Bedocchi for the function $s$. In the first section we have proved a similar result for the second function $t$, that is Lemma \ref{ab}. Lemma \ref{ab} allows us to get the implication from $b_{h+k}=b_{h}$ to $b_{h+k-1}=b_{h-1}$ also for $h$ even. In fact, if the preperiod length $h$ is even, both $h-1$ and $h+k-1$ are odd. Then $v_p(\overline{\alpha}_{h-1})=0$ and $v_p(\overline{\alpha}_{h+k-1})=0$. Therefore, if $b_{h-1}=t(\alpha_{h-1})$ and $b_{h+k-1}=t(\alpha_{h+k-1})$, then $b_{h-1},b_{h+k-1}\in K_p$. Reasoning as in the proof of Theorem \ref{LengthII} for the odd case, we get
\begin{align*}
    v_p(b_{h-1}-b_{h+k-1})_p&=v_p(\overline{\alpha}_{h-1}-\overline{\alpha}_{h+k-1})\geq \\ &\geq \min \{v_p(\overline{\alpha}_{h-1}),v_p(\overline{\alpha}_{h+k-1}) \}\geq 0.
\end{align*}
We conclude by Lemma \ref{ab} that $b_{h-1}=b_{h+k-1}$ also when $h$ is even.
\end{Remark}

We have seen in Remark \ref{pato} that using Lemma \ref{DisuII} and Lemma \ref{ab} we obtain that $b_{h+k}=b_{h}$ implies $b_{h+k-1}=b_{h-1}$ also for $h$ even when the two partial quotients are computed using the function $t$. The problem with \textit{Browkin II} is that for some odd $n\in\mathbb{N}$, \[b_n=t(\alpha_n)-sign(t(\alpha_n)).\]
This happens when $\alpha_n$ has not the constant term, in order to always recover a partial quotient with null $p$--adic valuation. In the following example we see that this case can actually occur.

\begin{Example}\label{ExaExp}
Let us consider
\[\sqrt{30}=3-3p+\ldots\in\mathbb{Q}_7,\]
then the expansion of $\alpha=\sqrt{30}+3$ is
\[ \sqrt{30}+3=\left[-1, \frac{3}{7}, 3, \frac{2}{7},\overline{1, \frac{2}{7}, -2, \frac{3}{7}, 1, \frac{2}{7}, 2, \frac{1}{7}, -1, -\frac{5}{7}}\right],\]
that is not purely periodic and has preperiod $4$.
Notice that the $p$--adic number $\alpha$ satistfies the hypotesis of Theorem \ref{LengthII} since

\begin{align*}
v_p\Big(3+\sqrt{30}\Big)&=0,\\
v_p\Big(3-\sqrt{30}\Big)&=v_p\Big(-3p+\ldots\Big)>0.
\end{align*}

In this case we can not make the step backward from $b_4=b_{14}$ to $b_3=b_{13}$ since
\begin{align*}
b_3&=t(\alpha_3)=\frac{2}{7},\\
b_{13}&=t(\alpha_{13})-sign(t(\alpha_{13}))=\frac{2}{7}-1=-\frac{5}{7}.
\end{align*}
In fact in the generation of $b_{13}$ it is used the sign function along with the $t$ function.

%However it is very interesting to notice that when the definition of the odd partial quotients $b_n$ is always the one in the second row of \eqref{Br2}, the procedure in the proof of \cref{LengthII} can be iterated to prove that $\alpha$ is immediately periodic. In fact,
\end{Example}

In Remark \ref{pato} and Example \ref{ExaExp}, we have observed that the converse of Theorem \ref{Bedo1} holds whenever the function $sign$ does not appear during the generation of the even partial quotients $b_h$ and $b_{h+k}$, or whenever they use the same function $sign$. It would be interesting, then, to understand for which of the quadratic irrationals of Theorem \ref{LengthII} these cases always occur, in order to prove an explicit characterization of the pure periodicity of \textit{Browkin II}.

%\begin{Corollary}
%Let $D\in\mathbb{Z}$ be not a square such that $\sqrt{D}\in\mathbb{Q}_p$,
%\[\sqrt{D}=a_0+a_1p+\ldots,\]
%with periodic \textit{Browkin II} expansion. Then, if $a_0\leq \frac{p-1}{4}$, the preperiod of $\sqrt{D}$ has length $1$.
%\begin{proof}
%By Proposition \ref{PureSqrts} we know that the $p$--adic continued fraction of $\sqrt{D}+a_0$ is purely periodic of the form
%\[\sqrt{D}+a_0=[\overline{2a_0,b_1,\ldots , b_{k-1}}]=[2a_0,\overline{b_1,\ldots , b_{k-1},2a_0}].\]
%Since $2a_0\leq\frac{p-1}{2}$, then it is a valid partial quotient and so the continued fraction of $\sqrt{D}$ is
%\[\sqrt{D}=[a_0,\overline{b_1,\ldots , b_{k-1},2a_0}],\]
%hence it is periodic with preperiod length $1$.
%\end{proof}
%\end{Corollary}

Now we investigate the length of the preperiod of \textit{Browkin II} expansions. In order to do that, we first introduce an algorithm that is a slight modification of \textit{Browkin II}.

%Now we introduce a slight modification of \textit{Browkin II}, useful to investigate the length of the preperiod of \textit{Browkin II} expansions.

\begin{Definition}[Browkin II*]
We call \textit{Browkin II*} the algorithm where the role of the functions $s$ and $t$ is switched. Starting from $\alpha_0\in\mathbb{Q}_p$, with $v_p(\alpha_0)<0$, the partial quotients of the $p$--adic continued fraction expansion are obtained for $n \geq 0$ by
\begin{align} 
\begin{cases}\label{Br2*}
b_n=s(\alpha_n) \ \ \ \ \ & \textup{if} \ n \ \textup{odd}\\
b_n=t(\alpha_n) & \textup{if} \ n \ \textup{even}\ \textup{and} \ v_p(\alpha_n-t(\alpha_n))= 0\\
b_n=t(\alpha_n)-sign(t(\alpha_n)) & \textup{if} \ n \ \textup{even} \ \textup{and} \ v_p(\alpha_n-t(\alpha_n))\neq 0\\
\alpha_{n+1}=\frac{1}{\alpha_n-b_n}.
\end{cases}
\end{align}

\end{Definition}

The $p$--adic convergence of any continued fraction of this kind is guaranteed. In fact, every infinite continued fraction $[b_0,b_1,b_2,\ldots]$ can be written as $b_0+\frac{1}{\alpha}$,
where $\alpha=[b_1,b_2,\ldots]\in\mathbb{Q}_p$ is obtained by \textit{Browkin II}. It is not hard to see that the observations in Remark \ref{RemConv} hold also for \textit{Browkin II*}, so the valuations of the convergents can be computed as in Equations $($\ref{Equ1}$)$ and $($\ref{Equ2}$)$.\\

In the following two theorems, we get some results similar to Theorem \ref{Bedo1} and Theorem \ref{LengthII} also for \textit{Browkin II*}.

\begin{Theorem}\label{Bedo1star}
If $\alpha\in\mathbb{Q}_p$ has a purely periodic continued fraction expansion $\alpha=[\overline{b_0,\ldots,b_{k-1}}]$ with \textit{Browkin II*}, then
\[v_p(\alpha)<0, \ \ v_p(\overline{\alpha})=0.\]
\begin{proof}
Let us assume that the expansion is purely periodic, that is
\[\alpha=[\overline{b_0,\ldots,b_{k-1}}],\]
with $k$ even by construction. Then
\[v_p(\alpha)=v_p(b_0)=v_p(b_k)<0.\]
For what concerns the valuation of the conjugate, by the pure periodicity we can write
\[\alpha=[b_0,\ldots,b_{k-1},\alpha].\]
Therefore
\[\alpha=\frac{\alpha A_{k-1}+A_{k-2}}{\alpha B_{k-1}+B_{k-2}},\]
and
\[B_{k-1}\alpha^2+(B_{k-2}-A_{k-1})\alpha -A_{k-2}=0.\]
Then
\[|\alpha \overline{\alpha}|_p=\frac{|A_{k-2}|_p}{|B_{k-1}|_p}=\frac{|b_0|_p\ldots |b_{k-2}|_p}{|b_1|_p\ldots |b_{k-1}|_p}=\frac{|b_0|_p}{|b_{k-1}|_p}.\]
Recalling that $|\alpha|_p=|b_0|_p$ and $|b_{k-1}|=1$ since $k-1$ is odd, then $|\overline{\alpha}|=1$ and the claim is proved, that is $v_p(\alpha)<0$ and $v_p(\overline{\alpha})=0$.
\end{proof}
\end{Theorem}

\begin{Theorem}\label{LengthStar}
Let $\alpha\in\mathbb{Q}_p$ with periodic \textit{Browkin II*} expansion
\[\alpha=[b_0,b_1,\ldots,b_{h-1},\overline{b_h,\ldots,b_{h+k-1}}].\]
Then, if
\[v_p(\alpha)<0, \ \ v_p(\overline{\alpha})=0,\]
the preperiod length can not be even.
\begin{proof}
Let us assume that $v_p(\alpha)<0$ and $v_p(\overline{\alpha})=0$ and that the expansion of $\alpha$ is periodic of the form
\[\alpha=[b_0,\ldots,b_{h-1},\overline{b_h,\ldots,b_{h+k-1}}].\]
By hypotesis, $v(b_0)=v(\alpha)<0$ and, by construction of \textit{Browkin II*} algorithm, for all $j\in\mathbb{N}$,
\begin{align*}
v_p(b_{2j})&<0,\\
v_p(b_{2j+1})&=0.
\end{align*}
So
\[v_p(\overline{\alpha}_1)=v_p\Big(\frac{1}{\overline{\alpha}-b_0}\Big)=-v_p(\overline{\alpha}-b_0)=-v_p(b_0)>0,\]
since $v_p(\overline{\alpha})=0$ and $v_p(b_0)<0$. Then
\[v_p(\overline{\alpha}_2)=v_p\Big(\frac{1}{\overline{\alpha}_1-b_1}\Big)=-v_p(\overline{\alpha}_1-b_1)=-v_p(b_1)=0,\]
since $v_p(\overline{\alpha}_1)=0$. It follows easily that, for all $j\in\mathbb{N}$,
\[v_p(b_{2j})<0, \ \ \ v_p(b_{2j+1})=0\]
Now let us suppose that $\alpha_h=\alpha_{h+k}$, that is
\[\overline{\alpha}_{h+k-1}-\overline{\alpha}_{h-1}=b_{h+k-1}-b_{h-1}.\]
If the preperiod length $h$ is even, both $h-1$ and $h+k-1$ are odd. Then $v_p(\overline{\alpha}_{h-1})>0$ and $v_p(\overline{\alpha}_{h+k-1})>0$, so:
\begin{align*}
v_p(b_{h-1}-b_{h+k-1})_p &=v_p(\overline{\alpha}_{h-1}-\overline{\alpha}_{h+k-1})\geq \\ &\geq \min \{v_p(\overline{\alpha}_{h-1}),v_p(\overline{\alpha}_{h+k-1}) \} > 0.
\end{align*}
We conclude by Lemma \ref{LemBedo} that $b_{h-1}=b_{h+k-1}$ and the claim is proved.
\end{proof}
\end{Theorem}

\begin{Example}
With these hypoteses, Theorem \ref{LengthStar} is the best we can obtain as a converse of Theorem \ref{Bedo1star}. In fact, if we consider the expansion of $\alpha=\frac{2+\sqrt{79}}{75}$ in $\mathbb{Q}_5$ is
\begin{align*}
\frac{2+\sqrt{79}}{75}=\Big[ &-\frac{7}{25}, 1, \frac{2}{5}, 2, -\frac{2}{5}, 1, \frac{1}{5}, 2, -\frac{4}{25}, 2, \frac{1}{5}, 1, -\frac{2}{5},\\
& 2, \frac{2}{5}, \overline{1, -\frac{7}{25}, -1, \frac{1}{5}, 2, \frac{9}{25}, -1, -\frac{3}{5}}\Big] ,
\end{align*}
that is not purely periodic and has preperiod $15$.
Notice that the $p$--adic number $\alpha$ satistfies the hypotesis of Theorem \ref{LengthStar} since starting from
\[\sqrt{79}=2+2p^2+p^3+\ldots\in\mathbb{Q}_5,\]
then:
\begin{align*}
v_p\Big(\frac{2+\sqrt{79}}{75}\Big)&=v_p(2+\sqrt{79})-v_p(75)=v_p(4+\ldots)-2=-2<0,\\
v_p\Big(\frac{2-\sqrt{79}}{75}\Big)&=v_p(2-\sqrt{79})-v_p(75)=v_p(2p^2+p^3+\ldots)-2=0.
\end{align*}
\end{Example}

 As consequence of the previous results, we are able to characterize the length of the preperiods for square roots of integers that have a periodic representation by means of \textit{Browkin II}.

\begin{Proposition}
\label{prop:preperiod}
Let $\sqrt{D}$ be defined in $\mathbb{Q}_p$, with $D\in\mathbb{Z}$ not a square, then, if $\sqrt{D}$ has  a periodic continued fraction with \textit{Browkin II}, the preperiod has length either one or even.
\begin{proof}
Notice that $\alpha=\sqrt{D}$ can not have purely periodic continued fraction, by Theorem \ref{Bedo1}. We can then write it as
\[\alpha=b_0+\frac{1}{\alpha_1}.\]
Since $\alpha$ has a periodic continued fraction, also $\alpha_1$ has periodic continued fraction. We are going to show that:
\begin{align*}
i) \ \  v_p(\alpha_1)=v_p\Big(\frac{1}{\alpha-b_0} \Big)=-v_p(\alpha-b_0)<0,\\
ii) \ \ v_p(\overline{\alpha}_1)=v_p\Big(\frac{1}{\overline{\alpha}-b_0} \Big)=-v_p(\overline{\alpha}-b_0)=0.
\end{align*}
Notice that $i)$ is satisfied since $s(\alpha)=b_0$ and $v_p(\alpha-b_0)>0$. Since $\overline{\alpha}=-\sqrt{D}$, then
\[\overline{\alpha}=-b_0+a'_1p+a_2'p^2+\ldots.\]
Now, $v_p(\overline{\alpha}-b_0)=0$ if and only if $b_0\neq-b_0$, that is $2b_0\neq 0$. This is always the case for $p\neq 2$, so by Theorem \ref{LengthStar}, $\beta$ can not have an even preperiod length. So $\alpha$ has either preperiod of length $1$ or of even length.
\end{proof}
\end{Proposition}

\section{Some periodic expansions}
\label{sec:perExp}

In \cite{BII}, Browkin characterized some expansions for square roots of integers provided by \textit{Browkin II} that have period $2$ and $4$. Through this construction, he proved the existence of infinitely many square roots of integers with periodic expansion of period $2$, similarly to what Bedocchi proved in \cite{BEII} for \textit{Browkin I}.\\
However, Browkin was not able to prove that the expansions having period of length $4$ that he provided were infinitely many, leaving open the problem of proving the existence of infinitely many square roots of integers having a periodic \textit{Browkin II} expansion with period of length $4$. Here we prove this result by constructing, for each prime $p$, an infinite class of square roots of integers that have a periodic \textit{Browkin II} continued fraction with the period of length $4$ in $\mathbb{Q}_p$. 

\begin{Theorem}
Given $D = \cfrac{1 - p^t}{(1-p)^2} \cdot p^2$, for any integer $t \geq 2$, then 
\[ \pm\sqrt{D} = \left[0, \pm \frac{1}{p}, \overline{\mp 1, \mp \frac{2(p^{t-1}-1)}{(p-1)p^{t-1}}, \mp 1, \pm \frac{2}{p}}\right]. \]
\end{Theorem}
\begin{proof}
In the following, we suppose $p = 4k + 1$, the proof for the case $p = 4k - 1$ is similar.\\
From \cite[Eq. 2.1]{BEII}, it follows that
\[\sqrt{D} = p(1 + p + \ldots + p^{t-1}) + Ap^{t+1},\]
where
\[ A = -\frac{p-1}{2} + p -\frac{p-1}{2} p^2 + p^3 - \ldots - \frac{p-1}{2}p^{t-1} + A' p^t, \]
supposing $t$ odd (a similar result holds for $t$ even).
Thus, considering $\alpha_0 = \sqrt{D}$, we immediately get $b_0 = s(\alpha_0) = 0$. Applying \emph{Browkin II}, we obtain
\[ \alpha_1 = \cfrac{\sqrt{D}}{D} = \cfrac{1}{p} - 1 + \cfrac{A p^{t+1}}{q\cdot p^2}, \]
where $q = \cfrac{1-p^t}{16k^2}$, and so $b_1 = t(\alpha_1) = \cfrac{1}{p}$.
The next complete quotient is
\[\alpha_2 = \cfrac{\sqrt{D} + qp}{1-q} = \cfrac{(2 - p)(1 - p^t) + A (p-1)^2 p^t}{p^{t-1}+p-2}.\]
Since, 
\[\cfrac{1}{p^{t-1}+p-2} = \cfrac{p-1}{2} + \ldots\]
we have $b_2 = s(\alpha_2) = -1$. In the next step, we have 
\begin{align*}
\alpha_3 &= \cfrac{p-\sqrt{D}}{p-\sqrt{D}+p\sqrt{D}} = - \cfrac{p^2 + \ldots + p^t + Ap^{t+1}}{(1-A+Ap)p^{t+1}} =\\
&= -\cfrac{1}{B}\left( \cfrac{1}{p^{t-1}}+ \ldots + \cfrac{1}{p} \right) - \cfrac{A}{B},
\end{align*}
where $B = 1 - A + Ap$ and $v_p(A) = v_p(B) = 0$. Now, we prove that $\cfrac{1}{B} = 2 + Cp^t$ for some $C$ such that $v_p(C) = 0$, in this way it follows that
\[ b_3 = t(\alpha_3) = -2\left( \cfrac{1}{p^{t-1}} + \ldots + \cfrac{1}{p} \right) = - \cfrac{p^{t-1} - 1}{2kp^{t-1}}, \]
considering that $4k = p-1$. For proving $\cfrac{1}{B} = 2 + Cp^t$, first of all we can observe that
\begin{align*}
B=&1 + \left(\cfrac{p-1}{2} - p + \cfrac{p-1}{2}p^2 - \ldots +\cfrac{p-1}{2} p^{t-1} + \ldots\right)+ \\
&+ \left(-\cfrac{p-1}{2} p + p^2 - \cfrac{p-1}{2}p^3 + \ldots - \cfrac{p-1}{2} p^{t} + \ldots\right). 
\end{align*}
Using that $1 + \cfrac{p-1}{2} = - \cfrac{p-1}{2} + p$, we obtain
\[ B = -\cfrac{p-1}{2} - \cfrac{p-1}{2} p - \ldots - \cfrac{p-1}{2}p^t + \ldots \]
Thus we surely have that $\dfrac{1}{B} = 2 + \ldots$ and we want to prove that $v_p\left(\cfrac{1}{B} - 2\right) \geq t+1$. For proving this, it is sufficient to observe that
\[ 1 - 2B = p^{t+1} + \ldots \]
Continuing to apply \emph{Browkin II}, with some calculations, we get
\[\alpha_4 = \cfrac{p^t(p-1)}{p^t - p - (p-1)\sqrt{D}}.\]
Exploiting the previous results, we have
\[\alpha_4 = -\cfrac{1}{1+Ap} \]
and considering that $\cfrac{-1}{1+Ap} = -1 +\ldots$, we obtain $b_4 = -1$. For the next step, we have
\[ \alpha_5 = \cfrac{p-p^t+(p-1)\sqrt{D}}{p-p^{t+1}+(p-1)\sqrt{D}} = 1+ \cfrac{1}{Ap} = \cfrac{2}{p} - 1 + \ldots\]
from which $b_5 = t(\alpha_5) = \cfrac{2}{p}$. Finally, one can check that \[\cfrac{1}{\alpha_5 - b_5} = \alpha_2,\]
and the  thesis follows.
\end{proof}

The previous theorem proves that there are infinitely many integers whose square roots have a periodic $p$--adic continued fraction expansion with period 4 by means of \textit{Browkin II}. Indeed, there are infinitely many $t$ such that $\cfrac{1-p^t}{16k^2}$, with $p = 4k \pm 1$, is an integer.

\section{Conclusions}
In this paper we have mainly analyzed the periodicity of \textit{Browkin II} that, at the state of the art, is the most close to a standard algorithm for continued fractions over the real field, in terms of its properties regarding finiteness and periodicity (since, as Browkin observed \cite{BII}, it appears to provide more periodic expansion for quadratic irrationals than \textit{Browkin I}).\\
In Section \ref{sec:perBro} we have found a necessary condition for the pure periodicity, that turns out to be not sufficient in general, in contrast on what Bedocchi proved for \textit{Browkin I} \cite{BEI}. On this purpose, Theorem \ref{LengthII} gives conditions to obtain the pure periodicity of the expansion in most cases. The motivations for the existence of rare exceptions are underlined in Remark \ref{pato} and Example \ref{ExaExp} and, in some sense, they implicitly characterize the pure periodicity of \textit{Browkin II}. It would be of great interest to find a full explicit characterizations for the $p$--adic numbers having a pure periodic \textit{Browkin II} continued fraction. 
Regarding the preperiod, Bedocchi proved in \cite{BEI}, that the lengths of the preperiods of \textit{Browkin I} expansions of square roots of integers could only be $2$ or $3$ and gave explicit conditions on when this happens. In the future, we aim to deepen the study of the preperiods for \textit{Browkin II} algorithm, trying to give more detailed conditions to distinguish the case of each possible preperiod, in view of \cref{prop:preperiod}.
\\
In Section \ref{sec:perExp} we have obtained infinitely many square roots of integers that have a periodic expansion by means of \textit{Browkin II} with period length $4$, solving the problem left open by Browkin in \cite{BII} and generalizing what he obtained in the same paper for the period length $2$. Moreover, through some experimental results, we believe that is possible to extend the result for all even period lengths. This computational result may be summarized in the following conjecture:
\begin{Conjecture}
For all even $h\in\mathbb{Z}$, there exist infinitely many $\sqrt{D}$, $D\in\mathbb{Z}$ not perfect square, such that \textit{Browkin II} continued fraction of $\sqrt{D}$ is periodic with period of length $h$.
\end{Conjecture}

\begin{comment}
\subsection{Considerazioni su infinite espansioni periodiche di interi di periodo fissato}
Vi riporto alcune considerazioni sul problema di trovare infinite espansioni di interi che siano periodiche di periodo $2$ con \textit{Browkin II}. Ho ragionato un po' al contrario partendo da espansioni del tipo $[b_0,\overline{\frac{b_1}{p^n},b_2}]$ per vedere quando convergessero a $\sqrt{m}$ con $m$ intero. Ho trovato che, ad esempio, \[\left[1,\overline{\frac{2}{p^n},2}\right]\]
converge sempre a $\sqrt{1+p^n}$. L'ho provato partendo da
\[\alpha=1+\cfrac{1}{\frac{2}{p^n}+\cfrac{1}{2+(\alpha-1)}},\]
facendo i calcoli si ottiene $\alpha^2=1+p^n$, cioè otteniamo infiniti $\alpha$ radici di interi con espansioni di periodo $2$.
Tuttavia, mi sono accorto successivamente che questo è un caso particolare del Lemma 2 dell'articolo di Browkin del 2000. Quindi, se non sto sbagliando qualcosa, già Browkin aveva provato che esistono infinite radici quadrate di periodo $2$ con \textit{Browkin II}. Browkin fa poi la stessa cosa nel Lemma $3$ per espansioni di periodo $4$, seguendo lo stesso ragionamento. Trova però condizioni più restrittive e lascia aperto il problema se esistano o no infiniti valori per i quali l'espansione risulta di preperiodo $1$ e periodo $4$. Fondamentalmente ragiona partendo da un'espansione periodica, che si riconduce ovviamente ad un'equazione del tipo
\[a\alpha^2+b\alpha+c=0,\]
e successivamente impone $b=0$ e $a|c$ in modo da trovare $\alpha$ radice quadrata di intero.
\end{comment}

\end{document}